\documentclass{svproc}
\usepackage{url}

\usepackage{amsmath}
\usepackage{amssymb}

\newcommand*\mycite[1]{\cite{#1}}

\newcommand*\LR{L(\mathbf R)}
\newcommand*\AD{\mathrm{AD}}
\newcommand*\Qa{{\mathbb Q_\alpha}}
\newcommand*\Qb{{\mathbb Q_\beta}}
\newcommand*\Qg{{\mathbb Q_\gamma}}
\newcommand*\Ql{{\mathbb Q_\lambda}}
\DeclareMathOperator\crit{crit}
\DeclareMathOperator\Ult{Ult}
\newcommand*\Ultra[2]{{\Ult({#1};{#2})}}

\newcommand*\comprh[2]{\left\{{#1} \mid {#2}\right\}}

\newcommand*\Gg{\Gamma_\gamma}
\newcommand*\Ggast{\Gamma^\ast_\gamma}
\newcommand*\Gglocal{\Gamma^{\text{local}}_\gamma}

\newcommand{\la}{\lambda}
\newcommand{\om}{\omega}
\newcommand{\bbQ}{\mathbb{Q}}
\newcommand{\seq}[1]{\langle {#1} \rangle}
\newcommand{\size}[1]{\left\vert {#1} \right\vert}
\newcommand{\cf}{\mathord{\mathrm{cf}}}
\newcommand{\dom}{\mathord{\mathrm{dom}}}

\begin{document}
\mainmatter              
\title{On Countable Stationary Towers}
\titlerunning{On Countable Stationary Towers}
%
\author{Yo Matsubara\inst{1} \and Toshimichi Usuba\inst{2}}
\authorrunning{Yo Matsubara and Toshimichi Usuba} 
%
\tocauthor{Yo Matsubara and Toshimichi Usuba}
\institute{Graduate School of Informatics, Nagoya University,\\
  Furo-cho, Chikusa-ku, Nagoya 464-8601, Japan\\
  \email{yom@math.nagoya-u.ac.jp}
  \and
  Faculty of Fundamental Science and Engineering, Waseda University,\\
  Okubo 3-4-1, Shinjyuku, Tokyo, 169-8555 Japan\\
  \email{usuba@waseda.jp}
}

\maketitle              

\begin{abstract}
  In this paper, we investigate properties of countable stationary towers.
  We derive the regularity properties of sets of reals in $\LR$ from some
  properties of countable stationary towers without explcit use of strong
  large cardinals such as Woodin cardinals.  We also introduce the notion of
  semiprecipitousness and investigate its relation to precipitousness and
  presaturation of countable stationary towers.  We show that precipitousness
  of countable stationary towers of weakly compact height implies the
  regularity properties of sets of reals in $\LR$.
\end{abstract}
\section{Introduction}
 In his expository article (written in Japanese) \mycite{Takeuti}, Gaisi Takeuti
called the development of a series of results concerning $\LR$ and large
cardinals, by Martin, Steel, and Woodin during the second half of the
1980's, a revolutionary leap in modern set theory.  Long before the 1980's,
Takeuti and Solovay pointed out that $\LR$, the smallest inner model
containing every real, is a natural candidate for an inner model of $\AD$
(axiom of determinacy).\mycite{Kanamori} What Takeuti called a revolutionary
leap has confirmed that their intuition was correct.

The first result in this revolutionary leap was proving various regularity
properties, such as Lebesgue measurability, the Baire property, and the
perfect set property, hold for every set of reals in $\LR$ from the existence
of a supercompact cardinal.  Soon after, the large cardinal hypothesis was
weakened and the conclusion was strengthened drastically.  For
convenience, throughout this paper, by regularity properties, we mean
Lebesgue measurability, the Baire property, and the perfect set property.

Woodin realized the existence of a certain generic elementary embedding
would imply regularity properties of sets of reals in $\LR$.\mycite{Woodin circulated notes}

\begin{theorem}[Woodin]\label{thm:Woodin}
  Let $\lambda$ be an inaccessible cardinal.
  Suppose there is a partial order $\mathbb P$ such that $\mathbb P$ forces the following:

  There is a generic elementary embedding $j : V \to M$ with
  \begin{enumerate}
  \item $\crit(j)=\aleph^V_1$ and $j(\aleph^V_1)=\lambda$
  \item $M$ is closed under taking $\omega$ sequences in $V^{\mathbb P}$
  \item every real is obtained by a small (size less than $\lambda$) forcing,
    i.e. in $V[G]$, where $G$ is $\mathbb P$-generic, for each real $x$ there are a poset $\mathbb P_x$ in the
    ground model $V$ with cardinality less than $\lambda$ and a
    $\mathbb P_x$-generic $H\in V[G]$ such that $x\in V[H]$.
  \end{enumerate}

  Then every set of reals in $\LR$ satisfies the regularity properties. \qed
\end{theorem}

Woodin developed the notion of stationary tower forcing to produce a
generic elementary embedding satisfying the hypothesis of the above
theorem.  Although Woodin's theory of stationary towers is more general,
we will concentrate our attention on countable stationary towers.  We
present a brief sketch of the notion of countable stationary tower forcing.  We refer the
reader to Larson's book \mycite{Larson} for a more thorough treatment of
stationary tower forcing.

\begin{definition}
  Let $\gamma$ be an uncountable limit ordinal.  We denote the set of all
  $p\in V_\gamma$ such that $p$ is a stationary subset of $[\bigcup p]^\om$ by $\Qg$.
  For $p,q\in\Qg$, we say $p\le_\Qg q$ if $\bigcup p\supseteq\bigcup q$ and
  $\forall x\in p(x\cap\bigcup q\in q)$.
\end{definition}

The partial order $\langle\Qg,\le_\Qg\rangle$ is known as a countable stationary tower of height
$\gamma$.  Throughout this paper, whenever we mention a stationary tower, we are
assuming its height is an uncountable limit ordinal.
 
If $G$ is a $\Qg$-generic filter over a ground model $V$, then we can take the
generic ultrapower $\Ultra V G$ of $V$.  We refer a reader to Larson's book
\mycite{Larson} for the definitions and basic facts concerning generic ultrapowers
induced by stationary tower forcing.

In general, there is no guarantee that the generic ultrapower $\Ultra V G$
is well-founded.

\begin{definition}
  $\Qg$ is said to be precipitous if $\Qg$ forces that $\Ultra V G$ is
  well-founded.
\end{definition}

\begin{remark}
  It turns out that the precipitousness of $\Qg$ is equivalent to ``$\Qg$
  forces that $\Ultra {V_\delta} G$ is well-founded where $\delta=(|V_\gamma|^+)^V$''. This motivated the
  next definition:
\end{remark}

\begin{definition}
  $\Qg$ is said to be semiprecipitous if $\Qg$ forces that $\Ultra {V_\gamma} G$ is well-founded.
\end{definition}

If $\Ultra V G$ is well-founded, then, by taking the transitive collapse of
$\Ultra V G$, we can define a generic elementary embedding $j : V \to M$ by
$j(x)=[c_x]$, where $c_x$ is the constant function whose domain is $[\aleph_1]^\om$ and that
takes the constant value $x$.

Woodin proved that if $\lambda$ is a large cardinal, such as a supercompact
cardinal, then $\Ql$ is precipitous.  Actually, he derived a somewhat stronger
property than precipitousness  from the large cardinal in his proof.

\begin{definition}
  $\Qg$ is said to be presaturated if $\Qg$ forces that $\Ultra V G$ is
  closed under taking $\omega$ sequences in generic extensions.
\end{definition}

\begin{remark}
  It is easy to see that presaturation of $\Qg$ implies
  precipitousness of $\Qg$.  It turns out that the presaturation of $\Qg$
  is equivalent to ``$\Qg$ forces $\cf(\gamma)>\omega$''.  Furthermore, the
  presaturation of $\Qg$ implies that $\Qg$ forces $j(\aleph^V_1)=\gamma$, where $j$
  is the corresponding generic elementary embedding.
\end{remark}

Woodin proved that if $\lambda$ is a supercompact cardinal, then $\Ql$
is presaturated and satisfies the hypotheses of Theorem~\ref{thm:Woodin}.
Therefore, the existence of a supercompact cardinal implies that
every set of reals in $\LR$ satisfies the regularity properties.

Woodin later proved that if $\lambda$ is a Woodin cardinal which is
substantially weaker than supercompact, then $\Ql$ is
presaturated.  
A generic elementary embedding generated by  $\Ql$ with Woodin $\lambda$ always satisfies
properties 1 and 2 of the hypotheses of Theorem 1.
But property 3 can fail.
It is known that the existence of one Woodin cardinal, in fact
even $\omega$ many Woodin cardinals, could not imply that every set of
reals in $\LR$ satisfies the regularity properties.  In particular J.\ Steel
\mycite{Steel} proved that if the existence of infinitely many
Woodin cardinals is consistent, then it is consistent with the
axiom of choice holding in $\LR$.  It turns out that if there are
infinitely many Woodin cardinals and a measurable cardinal
above them, then every set of reals in $\LR$ satisfies the regularity
properties.

In this paper, we investigate the following
questions:

\begin{enumerate}
\item Can we derive the regularity properties of sets of reals in $\LR$
  from properties of countable stationary towers, such as
  precipitousness or presaturation, without explicit use of strong
  large cardinals such as supercompact cardinals or Woodin
  cardinals?
\item Are there any relations between presaturation, precipitousness,
  and semiprecipitousness other than obvious implications?
\end{enumerate}

\section{Results}
Burke used the following game to characterize the
precipitousness of a stationary tower. \mycite{Burke}

For an uncountable limit ordinal $\gamma$,
let $\Gamma_\gamma$ denote the following two player game of length $\om$:
At each inning, players I and II choose conditions $p_n, q_n \in \bbQ_\gamma$ alternately with
$p_0 \ge q_0 \ge p_1 \ge q_1 \ge \cdots$.
For the play $\seq{p_n, q_n \mid n<\om}$,
player II wins if there is some $x \in [V_\gamma]^\om$
such that $x \cap \bigcup p_n \in p_n$ for every $n<\om$.
Otherwise I wins. The following theorem is due to Burke \mycite{Burke}:

\begin{theorem}\label{thm:Burke}
  Let $\gamma$ be an uncountable limit ordinal.
  Then the following are equivalent:
  \begin{enumerate}
  \item $\bbQ_\gamma$ is precipitous.
  \item Player I does not have a winning strategy in the game $\Gamma_\gamma$.
  \end{enumerate}
  \qed
\end{theorem}

For a certain $\gamma$, semiprecipitousness of $\bbQ_\gamma$ is equivalent to
precipitousness of $\bbQ_\gamma$.

\begin{proposition}\label{prop:Usuba03}
Let $\la$ be a singular cardinal such that $\size{V_\la}=\la$
and $\om<\cf(\la)<\la$.
If $\bbQ_\la$ is semiprecipitous, i.e. $\Vdash_{\bbQ_\la}$``\,$\mathrm{Ult}(V_\la; \dot G)$ is well-founded'',
then $\bbQ_\la$ is precipitous.
\end{proposition}
\begin{proof}
Suppose $\bbQ_\la$ is not precipitous.
Then player I has a winning strategy $\sigma$ in $\Gamma_\la$.

Let $T$ be the set of all finite sequences $\seq{\seq{q_i,x_i} \mid i \le n}$ 
such that:
\begin{enumerate}
\item $q_i \in \bbQ_\la$ and $x_i \in q_i$.
\item $x_{i+1} \cap \bigcup q_i =x_i$ for every $i <n$.
\item $\sigma(\emptyset) \ge q_0 \ge \sigma(q_0) \ge q_1 \ge \cdots \ge \sigma(q_0,\dotsc, q_{n-1} )\ge q_n$.
\end{enumerate}
For $s, t \in T$, we  define $\le_T$   by $s \le_T t$ if $s$ is an extension of $t$.
Since $\sigma$ is a winning strategy for player I,
$T$ is well-founded, i.e., there is no infinite descending sequence in $T$.

We note that $T \subseteq V_\la$.
Consider the structure $\seq{V_\la, \in, T}$,
where we identify $T$ as a unary predicate.
Since $\size{V_\la}=\la$, we can check the following:  
\begin{enumerate}
\item $\seq{V_\la, \in,T}\vDash$``$T$ is well-founded''.
\item Since $T \cap V_\alpha \in V_\la$ for every $\alpha<\la$,
we have $\seq{V_\la, \in,T} \vDash$``$\forall \alpha \exists x (x=T \cap V_\alpha)$''.
\item $\seq{V_\la, \in,T} \vDash$``$\forall \alpha, T \cap V_\alpha$ is well-founded,
and the rank function $f:T \cap V_\alpha \to \beta$ exists for some 
ordinal $\beta$''.
\end{enumerate}

Take a $\Ql$-generic $G$ with $\sigma(\emptyset) \in G$.
In $V[G]$, we can take the ultrapower of the structure $\seq{V_\la, \in, T}$
by $G$.
Since $\mathrm{Ult}(V_\la; G)$ is well-founded,
the ultrapower of $\seq{V_\la, \in, T}$
is also well-founded.
Let $M=\seq{M, \in, j(T)}$ be the transitive collapse of the ultrapower,
and $j:V_\la \to M$ be the elementary embedding.
Then, because $\om<\cf(\la)<\la$,
we have that $j``\la$ is bounded in $M \cap\mathrm{ON}$.
Let $\gamma=\sup(j``\la) \in M \cap \mathrm{ON}$.
Consider $T^*=j(T) \cap M_\gamma$.
By (2) above, we have $T^* \in M$,
and by (3), there exists the rank function $f:T^* \to \beta$ in $M$.
Hence, $T^*$ is also well-founded in $V[G]$,
and there is no infinite descending sequence of $T^*$ in $V[G]$.

By induction on $n<\om$,
we define a descending sequence $\seq{p_n,q_n \mid n<\om}$ as follows:
First, let $p_0 =\sigma(\emptyset) \in G$.
Then $\{\sigma(q) \mid q \le p_0\}$ is dense below $p_0$,
so we can take $q_0 \le p_0$ and $p_1=\sigma(q_0) \le q_0$ with
$p_1 \in G$. We have $q_0 \in G$.
Again, since $\{\sigma(q_0,q) \mid q \le p_1\}$ is dense below $p_1$,
we can choose $q_1 \le p_1$ and $p_2=\sigma(q_0,q_1) \in G$,
and so on.
We know $j(p_n)=j(\sigma)(j(q_0),\dotsc, j(q_{n-1})) \ge j(q_n)$.
For $n<\om$, let $x_n=\bigcup q_n$.
We know $j``x_n \in j(q_n)$,
and $j``x_{n+1} \cap \bigcup(j(q_n))=j``x_n$.
Because $j(q_n), j``x_n \in M_\gamma$,
we have that $\seq{\seq{j(q_i),j``x_i} \mid i \le n} \in T^*$ for every $n<\om$.
Therefore, the sequence $\seq{\seq{j(q_n), j``x_n} \mid n<\om}$ induces an infinite descending sequence in $T^*$.
This is a contradiction.
\qed
\end{proof}

\begin{remark}
  The reader might wonder if the hypothesis of the last
  proposition is vacuous.  In our next paper \mycite{Matsubara-Usuba}, we
  prove the consistency of precipitousness of $\Ql$ where $\lambda$ satisfies
  the conditions of Proposition~\ref{prop:Usuba03} assuming a sufficiently strong
  large cardinal hypothesis.
\end{remark}

Later in this paper, we will show that, in general,
semiprecipitousness of $\Ql$ does not imply precipitousness.

We can also use Burke's game to prove the following proposition:

\begin{proposition}
  Suppose $\alpha$ and $\beta$ are uncountable limit ordinals
  with $\alpha<\beta$.  If $\Qa$ is precipitous, then, for each $p\in\Qa$, there
  is some $p^\ast\in\Qb$ such that $p^\ast\le p$ and $p^\ast$ $\Vdash_{\bbQ_\beta}$
  ``$\dot G\cap\Qa$ is $\Qa$-generic'' where $\dot G$ is the canonical $\Qb$-name for a generic filter.
\end{proposition}

This proposition is an immediate consequence of the next
lemma. The referee informed us that this lemma is a  ``tower" analogue of Proposition 2.4
of 
\mycite{Ketchersid-Larson-Zapletal}.

\begin{lemma}\label{lem:Usuba04}
Let $\alpha$ be an uncountable  limit ordinal,
and suppose $\bbQ_\alpha$ is precipitous.
Then, for every $p \in \bbQ_\alpha$,
the set $\{x \in [V_{\alpha+1}]^\om \mid x \cap \bigcup p \in p,
\forall D \in x\,($if $D \subseteq \bbQ_{\alpha}$ is dense, then $\exists r \in D \cap x \,(x \cap \bigcup r \in r))\}$
is stationary in $[V_{\alpha+1}]^\om$.
\end{lemma}

\begin{proof}
Take $p \in \bbQ_\alpha$.
Fix a well-ordering $\Delta$ on $V_{\alpha+1}$
and a function $ f: [V_{\alpha+1}]^{<\om} \to V_{\alpha+1}$.
For a set $x \subseteq V_{\alpha+1}$, let $Sk(x)$ denote the Skolem hull of
$x$ under the structure $\seq{V_{\alpha+1}, \in, \Delta, f}$.
We will find a set $x$ such that $Sk(x)$ belongs to the required set.

Now we describe a strategy for player I in $\Gamma_\alpha$.
Fix a surjection $\pi:\om \to \om \times \om$
such that if $\pi(n)=\seq{i,j}$ then $i \le n$.
For a countable $x \subseteq V_{\alpha+1}$, 
we also fix an enumeration $\seq{D^x_j \mid j<\om}$
of all dense subsets of $\bbQ_\alpha$ that belong to $Sk(x)$.

First, let $p'_0 = p$.
Take $j<\om$ with $\pi(0)=\seq{0,j}$.
Since $p_0'$ is stationary in $[\bigcup p'_0]^\om$ ,
there is some dense set $D_0 \subseteq \bbQ_\alpha$
with $\{x \in p_0' \mid D_0=D^x_j\}$ stationary in $[\bigcup p'_0]^\om$ ,
so
$\{x \in p_0' \mid D_0=D^x_j\}$ is a condition of $\bbQ_\alpha$.
 Take $p_0'' \in D_0$ with
$p_0'' \le \{x \in p_0' \mid D_0=D^x_j\}$.
Then take $p_0 \le p_0''$ such that
$p_0'' \in x$ and $Sk(x) \cap \bigcup p_0=x$
for every $x \in p_0$.
This $p_0$ is the first move of player I.

Let $n>\om$,
and suppose
$p \ge p_0' \ge p_0'' \ge p_0 \ge q_0 \ge \cdots \ge q_{n-1}$ were chosen so that,
for every $k <n$:
\begin{enumerate}
\item If $\pi(k)=\seq{i,j}$, 
then, for every $x \in p_k''$,
we have $D_{j}^{x \cap \bigcup p_i'}=D_k$.
\item $p_k'' \in D_k$.
\item $Sk(x) \cap \bigcup p_k=x$ and $p_k'' \in x$ for every $x \in p_k$.
\end{enumerate}
As in the $n=0$ case,
we can choose $p_n \le p_n'' \le p_n' \le q_{n-1}$ as required,
and player I chooses $p_n$ as his move.
This completes the description of our strategy for player I.

Since $\bbQ_\alpha$ is precipitous,
the strategy above is not a winning strategy for player I.
Thus, we can find a descending sequence
$p \ge p_0' \ge p_0'' \ge p_0 \ge q_0 \ge \cdots \ge q_{n-1} \ge p_n' \ge \cdots$
such that each $p_n, p_n', p_n''$  is chosen according to the above strategy,
but there is some $x \subseteq V_\alpha$
such that $x \cap \bigcup p_n \in p_n$ for every $n<\om$.
We may assume that $x=\bigcup_{n<\om}( x \cap \bigcup p_n)$.
Since $p'_n \ge p_n'' \ge p_n \ge p'_{n+1}$,
we also have
$x \cap \bigcup p_n' \in p_n'$, $x \cap \bigcup p_n'' \in p_n''$,
and $x=\bigcup_{n<\om} (x \cap \bigcup p'_n)=
\bigcup_{n<\om} (x \cap \bigcup p''_n)$.
Hence $Sk(x)=\bigcup_{n<\om} Sk(x \cap \bigcup p'_n)$.
For $n<m$, we have $Sk(x \cap \bigcup p_m) \cap \bigcup p_n''=(x \cap \bigcup p_m) \cap \bigcup p_n''=x \cap \bigcup p_n''$,
thus $Sk(x) \cap \bigcup p_n''=x \cap \bigcup p_n''$ for every $n<\om$.
We show that, if $D \in Sk(x)$ is dense in $\bbQ_\alpha$,
then there is $r \in D \cap Sk(x)$ with $Sk(x) \cap \bigcup r \in r$,
which completes our proof.
Take a dense set $D \in Sk(x)$.
Since $ Sk(x)=\bigcup_{n<\om} Sk(x \cap \bigcup p'_n)$,
there is some $i<\om$
with $D \in Sk(x \cap \bigcup p_i')$.
Take $j<\om$ with $D=D^{x \cap \bigcup p_i'}_j$,
and pick $n<\om$ with $\pi(n)=\seq{i,j}$.
Then, since $x \cap \bigcup p_n'' \in p_n''$,
we have that $p_n'' \in D_n=D^{x \cap \bigcup p_n'}_j=D$.
$p_n'' \in x \cap \bigcup p_n \subseteq x$,
so $p_n'' \in D \cap Sk(x)$ and $Sk(x) \cap \bigcup p_n''=x \cap \bigcup p_n'' \in p_n''$.
\qed
\end{proof}

Using the above proposition we can obtain the next result
concerning Question~1.

\begin{proposition}\label{prop:katakana-a}
  Suppose $\lambda$ is an inaccessible cardinal.  If $\Ql$ is presaturated and
  $\comprh{\alpha<\lambda}{\text{$\Qa$ is precipitous}}$ is cofinal in $\lambda$, then every set of
   reals in $\LR$ satisfies the regularity properties.
\end{proposition}

\begin{proof}
  For each $\gamma < \lambda$, we let $E_\gamma$ denote the set of all of  $p \in \bbQ_\lambda$ such that 
  $p$ forces ``$\exists \alpha  (\gamma < \alpha < \lambda  \land \text{$G\cap\Qa$ is $\Qa$-generic} )$ where  $G$ is $\Ql$-generic''.
  We note that Proposition 2 implies that, for each $\gamma <\lambda$, the set $E_{\gamma}$ is dense in $\Ql$.
  Let $G$ be a $\Ql$-generic filter over $V$.  We work in $V[G]$.
  From the denseness of $E_\gamma$ for all $\gamma < \lambda$, 
  it is easy to see that
  $\comprh{\alpha<\lambda}{\text{$G\cap\Qa$ is $\Qa$-generic}}$ is cofinal.

  If $j:V\to M$ is the corresponding generic elementary embedding,
  then conditions (1) and (2) of Woodin's theorem are satisfied
  since $\Ql$ is presaturated.  Now we check that condition (3) is also
  satisfied.
  
  Suppose $r$ is a real in $V[G]$.  By presaturation of $\Ql$ we know
  that $r$ belongs to $M$, the generic ultrapower.  Hence there is some
  function $f$ in the ground model with $\dom(f)\in G$ such that
  $[f]=r$.  Choose large enough $\alpha<\lambda$ such that $\dom(f)\in\Qa$
  and $G\cap\Qa$ is $\Qa$-generic.  This would imply that $r\in V[G\cap\Qa]$.
  Hence the conditions of Woodin's theorem are satisfied by our
  generic elementary embedding. \qed 
\end{proof}

The stationarity of $\comprh{\alpha<\lambda}{\text{$\Qa$ is precipitous}}$ has some
interesting consequences.

\begin{lemma}\label{lem:Usuba06}
Suppose $\la$ is inaccessible.
If the set $\{\alpha<\la \mid \bbQ_\alpha$ is precipitous$\}$
is stationary in $\la$,
then $\bbQ_\la$ is presaturated.
\end{lemma}
\begin{proof}
Let $p \in \bbQ_\la$, and
take $\bbQ_\la$-names $\dot f_n$ ($n<\om$) such that
$p \Vdash_{\bbQ_\la}$``$\dot f_n \in V$ is a function whose domain is some element of  $\Ql$''.
For each $n<\om$, take a maximal antichain $A_n \subseteq \bbQ_\la$
such that, for each $q \in A_n$, either $q \bot p$ or $q \le p$, and, if
$q \le p$, then there is some function $f^q_n$ with $q \Vdash_{\bbQ_\la}$``$\dot f_n=f^q_n$''.
Since the set $\{\alpha<\la \mid \bbQ_\alpha$ is precipitous$\}$ is stationary in $\la$,
there is $\alpha<\la$ such that
$\bbQ_\alpha$ is precipitous, and, for every $n<\om$,
$A_n \cap \bbQ_\alpha$ is predense in $\bbQ_\alpha$, and
$\dom(f^q_n) \in V_\alpha$ for every $q \in A_n \cap V_\alpha$. 
Let $D_n=\{r \in \bbQ_\alpha \mid  \exists q \in A_n \cap \bbQ_\alpha \,(r \le q)\}$.
$D_n$ is dense in $\bbQ_\alpha$.

By the previous lemma,
the set $p'=\{x \in [V_{\alpha+1}]^\om \mid
x \cap \bigcup p \in p,
\forall n<\om
\exists r \in D_n \cap x \,(x \cap \bigcup r \in r)\}$
is stationary in $[V_{\alpha+1}]^\om$.
We know $p' \le p$.
For $x \in p'$ and $n<\om$, take $r^x_n \in D_n \cap x$
with $x \cap \bigcup r^x_n \in r^x_n$,
and we may assume that
there is a unique $q^x_n \in A_n \cap x$ with $r^x_n \le q^x_n$.
Then define $g:[V_{\alpha+1}]^\om \to V$ by
$g(x)=\{f_n^{q^x_n}(x \cap \bigcup q^x_n) \mid n<\om\}$.
It is routine to check that
$p' \Vdash_{\bbQ_\la}$``$[g]$ represents the set $\{[\dot f_n] \mid n<\om\}$''.
\qed
\end{proof}

We use the last lemma to obtain some consequences of
precipitousness of $\Ql$ where $\lambda$ is a weakly compact cardinal.

\begin{proposition}\label{prop:katakana-i}
Suppose $\la$ is a weakly compact cardinal.
If $\bbQ_\la$ is precipitous,
then $\{\alpha<\la \mid \bbQ_\alpha$ is precipitous$\}$ is stationary in $\la$.
Hence , under our hypothesis, $\bbQ_\la$ is presaturated.
\end{proposition}
\begin{proof}
The statement that ``player I does not have a winning strategy for Burke's game $\Gamma_\lambda$''
can be described by some $\mathrm\Pi^1_1$-statement $\varphi$
over $\seq{V_\la, \in}$.
Hence, if $\bbQ_\la$ is precipitous,
then $V_\la \vDash \varphi$,
and, since $\la$ is weakly compact,
we have that $\{\alpha<\la \mid V_\alpha \vDash \varphi\}$
is stationary in $\la$.
Therefore 
the set $\{\alpha<\la \mid \bbQ_\alpha$ is precipitous$\}$ is stationary in $\la$.
\qed
\end{proof}

By combining Proposition~\ref{prop:katakana-a} and Proposition~\ref{prop:katakana-i} we obtain
the next result:

\begin{corollary}
  If $\Ql$ is precipitous for a weakly compact cardinal $\lambda$,
  then every set of reals in $\LR$ satisfies the regularity properties.
\end{corollary}

We note that the least Woodin cardinal is not a weakly compact cardinal. A weaker large cardinal could also reflect semistationarity.

\begin{proposition}\label{prop:katakana-u}
  If $\Ql$ is semiprecipitous for a Mahlo cardinal $\lambda$,
  then the set $\comprh{\alpha<\lambda}{\text{$\Qa$ is semiprecipitous where $\alpha$ is regular }}$ is stationary in $\lambda$.
\end{proposition}

Before we present the proof of this proposition, we state its
corollary:

\begin{corollary}\label{temp}
  The semiprecipitousness of $\Qg$, in general, cannot
  imply the precipitousness of $\Qg$.
\end{corollary}

\begin{proof}
  Suppose otherwise.  Assume that $\lambda$ is a Woodin cardinal.
  By Woodin's result, we know that $\Ql$ is presaturated.

  Since a Woodin cardinal is Mahlo, by Proposition~\ref{prop:katakana-u}, $\comprh{\alpha<\lambda}{\text{$\Qa$ is semiprecipitous}}$ is
  stationary.  So, by our assumption, $\comprh{\alpha<\lambda}{\text{$\Qa$ is precipitous}}$ is
  stationary.  Then, by Proposition~\ref{prop:katakana-a}, every set of reals in $\LR$
  satisfies the regularity properties.  But the existence of one Woodin cardinal
  cannot imply this conclusion. \qed 
\end{proof}

To prove Proposition~\ref{prop:katakana-u}, we use the following variant of Burke's
game $\Gg$ to characterize the precipitousness of $\Qg$.

Let $\Ggast$ denote the following two player game of length $\omega$:
At the $n$-th inning of the game player I plays $\langle p_n,f_n\rangle$ with
$p_n\in\Qg$, $f_n:p_n\to\mathrm{ON}$, and player II plays $q_n\in\Qg$ alternately with
the requirements: $p_0\ge q_0\ge p_1\ge q_1\ge\cdots$ and
$\forall x\in p_{n+1}(f_{n+1}(x)\in f_n(x\cap\bigcup p_n))$.  Player I wins if and only if the
game does not stop.  By the proof of Burke's theorem, we see that
$\Qg$ is precipitous if and only if player I does not have a winning
strategy.  This variant of Burke's game was inspired by a similar
game due to Goldring \mycite{Goldring} characterizing
precipitousness of ideals over $\mathcal P_\kappa\lambda$.  Note that one advantage of the
game $\Ggast$ over $\Gg$ is that, if player I does not have a winning
strategy in $\Ggast$, then, by the theorem of Gale and Stewart, player II does have a winning strategy.

\begin{proof}[of Proposition~\ref{prop:katakana-u}]
  Suppose that $\Ql$ is semiprecipitous for a Mahlo cardinal $\lambda$.
  For an uncountable limit ordinal $\gamma$, we consider the following
  game $\Gglocal$ which is a localized version of $\Ggast$.

  The game $\Gglocal$ is almost identical to $\Ggast$ with the added
  requirement that the function $f_n$, played by player I, must take
  its values in $\gamma$.  It is easy to see that player I does not have a
  winning strategy in $\Gglocal$ is equivalent to $\Qg$ is semiprecipitous.
  Since player I does not have a winning strategy for the game
  $\Gamma^{\text{local}}_\lambda$, player II does have a winning strategy $\tau$ for that game.  From the
  Mahloness of $\lambda$, by letting player II play the strategy $\tau$, it is not
  too difficult to see that the set
  $\comprh{\gamma<\lambda}{\text{$\tau$ is a winning strategy for player II in $\Gglocal$ with $\gamma$ regular}}$  is stationary
  in $\lambda$.
  \qed
\end{proof}

 It turns out that a  large cardinal which is weaker than  Mahlo suffices to reflect semiprecipitousness to points of cofinality $\omega$.
\begin{proposition}\label{prop:note-u}
  If $\Ql$ is semiprecipitous for an inaccessible cardinal $\lambda$,
  then there is a club $C$ in $\la$
  such that, for every $\gamma \in C$,
  if $\cf(\gamma)=\om$ then $\bbQ_\gamma$ is semiprecipitous.
\end{proposition}
\begin{proof}
Since $\bbQ_\la$ is semiprecipitous,
player II in $\Gamma^{\mathrm{local}}_\la$ has a winning strategy $\tau$.

Take a sufficiently large regular cardinal, and take an elementary submodel $M \prec H_\theta$ containing
all relevant objects satisfying $M \cap \la \in \la$ and $\cf(M \cap \la)=\om$.
Let $\gamma=M \cap \la$. We shall show that player II in $\Gamma^{\mathrm{local}}_\gamma$
has a winning strategy. Let us describe player II's strategy.

Suppose player I plays $\seq{p_0,f_0}$.
Note that $p_0 \in \bbQ_\gamma \subseteq M$ but $f_0$ may not be an element of $M$,
since $\mathrm{range}(f_0)$ could be cofinal in $\gamma$.
However, since $\cf(\gamma)=\om$,
there is $\gamma'<\gamma$ such that
$\{x \in p_0\mid f(x) \le \gamma'\}$ is stationary.
Let $\gamma*$ be the least ordinal such that $p_0'=\{x \in p_0 \mid f(x) \le \gamma*\} \in \bbQ_\gamma$.
Then we have $\seq{p_0', f_0 \restriction p_0'} \in M$,
and we can identify $\seq{p_0', f_0 \restriction p_0'}$ as player I's play in
$\Gamma^{\mathrm{local}}_\la$.
Let $q_0=\tau(\seq{p_0', f_0 \restriction p_0'}) \in M \cap \bbQ_\la=\bbQ_\gamma$
be player II's  response. Note that $q_0 \le p_0' \le p_0$ in $\bbQ_\gamma$.
If $\seq{p_1, f_1}$ is a response of player I in $\Gamma^{\mathrm{local}}_\gamma$,
by repeating the above procedure, we can find $p_1' \subseteq p_1$ with
$\mathrm{range}(f_1 \restriction p_1')$  bounded in $\gamma$.
Since $\bigcup p_0=\bigcup p_0'$,
for every $x \in p_1'$, we have $f_1(x) \in f_0(x \cap \bigcup p_0)=f_0(x \cap \bigcup p_0')$.
Hence $\seq{p_1',f_1\restriction p_1'}, q_0, \seq{p_0',f_0\restriction p_0'}$ satisfy
the requirements of a play in $\Gamma^{\mathrm{local}}_\lambda$.
Since $\seq{p_0', f_0 \restriction p_0'},\seq{p_1',f_1\restriction p_1'} \in M$,
we let player II play  $q_1=\tau(\seq{p_0', f_0 \restriction p_0'},\seq{p_1',f_1\restriction p_1'}) \in M$,
and so on.
Since $\tau$ is a winning strategy of player II in $\Gamma^{\mathrm{local}}_\la$,
it is easy to see that the strategy above is a winning strategy of player II in $\Gamma^{\mathrm{local}}_\gamma$.
\qed
\end{proof}

\par
\begin{remark}
We note that, by the proof of Corollary~\ref{temp}, the last proposition indicates that the hypothesis ``$\om<\cf(\la)$'' of Proposition~\ref{prop:Usuba03}  cannot be dropped.
\end{remark}

In the late 1980's, the first author was in Los Angeles and
received a letter from Gaisi Takeuti.  He was inquiring about results
concerning $\LR$ and large cardinals.  The first author still
remembers the letter that professed Takeuti's passion and
enthusiasm for set theory and logic.  We regret that we did not
have an opportunity to ask Professor Takeuti's view on the
revolutionary leap in person.

\subsection*{Acknowledgments}
The first author's research was supported by JSPS KAKENHI Grant Nos. 18K03394.
The second author's research was supported by JSPS KAKENHI Grant Nos. 18K03403 and 18K03404.

%
%


\begin{thebibliography}{9}
%
\bibitem {Burke}
  Burke, D.\ R.\ : Precipitous towers of normal filters.
  The Journal of Symbolic Logic, 62(3): 741--754, (1997).

\bibitem {Goldring}
  Goldring, N.\ : Projecting precipitousness. Israel Journal of Mathematics,
  74(1): 13--31, (1991).

\bibitem {Kanamori}
  Kanamori, A.\ :
  The Higher Infinite. Large Cardinals in Set Theory from their Beginnings.
  Perspectives in mathematical logic.
  Springer-Verlag, Berlin, Heidelberg, New York, etc., (1994).

\bibitem {Ketchersid-Larson-Zapletal}
Ketchersid, R., Larson, P., Zapletal, J.\ : Increasing ${ \delta  }_{ 2 }^{ 1 }$ and Namba-style forcing.
The Journal of Symbolic Logic, 72(4): 1372--1378, (2007).

\bibitem {Larson}
  Larson, P.\ : The Stationary Tower: Notes on a Course by W.\ Hugh Woodin,
  University Lecture Series, vol.~32 (American Mathematical Society),
  Oxford University Press, Oxford (2004).

\bibitem {Matsubara-Usuba}
  Matsubara, Y., Usuba, T.\ : Reflection principle and stationary towers, in preparation.

\bibitem {Steel}
  Steel, J.\ R.\ : Inner models with infinitely many Woodin cardinals.
  Annals of Pure and Applied Logic, 65(2): 185--209, (1993).

\bibitem {Takeuti}
  Takeuti, G.\ : Gendai Shugoron Nyumon (Introduction to Modern
  Set Theory), revised edition. Nipponhyoronsha, Tokyo (1989).

\bibitem {Woodin circulated notes} 
  Woodin, W.\ H.\ : $\mathrm\Sigma^2_1$ absoluteness. Circulated Notes (1985, May).
\end{thebibliography}
\end{document}